\documentclass[11pt]{amsart}
\usepackage{amsmath}
\usepackage{amssymb}
\usepackage{amscd}
\usepackage{amsthm}

\usepackage[all]{xy}

\usepackage[usenames,dvipsnames]{color}

\newtheorem{thm}{Theorem}[section]
\newtheorem{lem}[thm]{Lemma}
\newtheorem{cor}[thm]{Corollary}
\newtheorem{conj}[thm]{Conjecture}
\newtheorem{prop}[thm]{Proposition}
\newtheorem{q}[thm]{Question}

\newtheorem{ex}[thm]{Example}

\theoremstyle{remark}

\newtheorem{rem}[thm]{Remark}

\theoremstyle{definition}

\numberwithin{equation}{section}

\allowdisplaybreaks[4]

\DeclareMathOperator{\rank}{rank}

\DeclareMathOperator{\Ext}{Ext}
\DeclareMathOperator{\Pic}{Pic}

\DeclareMathOperator{\CH}{CH}

\DeclareMathOperator{\rat}{rat}

\begin{document}

\vfuzz0.5pc
\hfuzz0.5pc 

\newcommand{\claimref}[1]{Claim \ref{#1}}
\newcommand{\thmref}[1]{Theorem \ref{#1}}
\newcommand{\propref}[1]{Proposition \ref{#1}}
\newcommand{\lemref}[1]{Lemma \ref{#1}}
\newcommand{\coref}[1]{Corollary \ref{#1}}
\newcommand{\remref}[1]{Remark \ref{#1}}
\newcommand{\conjref}[1]{Conjecture \ref{#1}}
\newcommand{\questionref}[1]{Question \ref{#1}}
\newcommand{\defnref}[1]{Definition \ref{#1}}
\newcommand{\secref}[1]{Sec. \ref{#1}}
\newcommand{\ssecref}[1]{\ref{#1}}
\newcommand{\sssecref}[1]{\ref{#1}}

\newcommand{\red}{{\mathrm{red}}}
\newcommand{\tors}{{\mathrm{tors}}}
\newcommand{\eq}{\Leftrightarrow}

\newcommand{\mapright}[1]{\smash{\mathop{\longrightarrow}\limits^{#1}}}
\newcommand{\mapleft}[1]{\smash{\mathop{\longleftarrow}\limits^{#1}}}
\newcommand{\mapdown}[1]{\Big\downarrow\rlap{$\vcenter{\hbox{$\scriptstyle#1$}}$}}
\newcommand{\smapdown}[1]{\downarrow\rlap{$\vcenter{\hbox{$\scriptstyle#1$}}$}}

\newcommand{\A}{{\mathbb A}}
\newcommand{\I}{{\mathcal I}}
\newcommand{\J}{{\mathcal J}}
\newcommand{\CO}{{\mathcal O}}
\newcommand{\CC}{{\mathcal C}}
\newcommand{\C} {{\mathbb C}}
\newcommand{\BC}{{\mathbb C}}
\newcommand{\BQ}{{\mathbb Q}}
\newcommand{\m}{{\mathcal M}}
\newcommand{\h}{{\mathcal H}}
\newcommand{\ZZ}{{\mathcal Z}}
\newcommand{\Z} {{\mathbb Z}}
\newcommand{\BZ}{{\mathbb Z}}
\newcommand{\W}{{\mathcal W}}
\newcommand{\Y}{{\mathcal Y}}
\newcommand{\T}{{\mathcal T}}
\newcommand{\BP}{{\mathbb P}}
\newcommand{\CP}{{\mathcal P}}
\newcommand{\G}{{\mathbb G}}
\newcommand{\BR}{{\mathbb R}}
\newcommand{\D}{{\mathcal D}}
\newcommand {\LL} {{\mathcal L}}
\newcommand{\f}{{\mathcal F}}
\newcommand{\E}{{\mathcal E}}
\newcommand{\BN}{{\mathbb N}}
\newcommand{\N}{{\mathcal N}}
\newcommand{\K}{{\mathcal K}}
\newcommand{\R} {{\mathbb R}}
\newcommand {\PP} {{\mathbb P}}
\newcommand{\BF}{{\mathbb F}}
\newcommand{\closure}[1]{\overline{#1}}
\newcommand{\EQ}{\Leftrightarrow}
\newcommand{\imply}{\Rightarrow}
\newcommand{\isom}{\cong}
\newcommand{\embed}{\hookrightarrow}
\newcommand{\tensor}{\mathop{\otimes}}
\newcommand{\wt}[1]{{\widetilde{#1}}}
\newcommand{\ol}{\overline}
\newcommand{\ul}{\underline}
\newcommand{\QQ}{{\mathcal Q}}

\newcommand{\bs}{{\backslash}}
\newcommand{\CS}{{\mathcal S}}
\newcommand{\Q} {{\mathbb Q}}
\newcommand {\F} {{\mathcal F}}

\author{Xi Chen}
\address{632 Central Academic Building\\ University of Alberta\\ Edmonton, Alberta T6G 2G1, CANADA}
\email{xichen@math.ualberta.ca}

\author{James D. Lewis}
\address{632 Central Academic Building\\ University of Alberta\\ Edmonton, Alberta T6G 2G1, CANADA}
\email{lewisjd@ualberta.ca}

\keywords{$K3$ surface, rational curve, elliptic curve, elliptic surface}

\subjclass{Primary 14N10, 14J28}

\renewcommand{\abstractname}{Abstract}
\begin{abstract}
Using the dynamics of self rational maps of elliptic $K3$ surfaces
together with deformation theory,
we prove that the union of rational curves is dense on a very general 
$K3$ surface and that the union of elliptic curves is dense in the 1st jet space of a 
very general $K3$ surface, both in the strong topology. 
\end{abstract}

\thanks{Both authors partially supported by a grant from the Natural Sciences and Engineering Research Council of Canada.}

\date{\today}

\title{Density of rational curves on $K3$ surfaces}

\maketitle

\section{Introduction}

\subsection{Density of rational curves}

The main purpose of this note is to prove that 
the union of all rational curves on a ``very general'' projective $K3$ surface $X$ is dense in the usual topology.
Here ``very general'' takes some explanation. It is weaker than the usual sense of being in
the complement of  countably many closed proper subvarieties.

Let $\K_g$ be the moduli space of $K3$ surfaces of genus $g\ge 2$ and $\CS_g$ be the universal family over $\K_g$. That is,
\begin{equation}\label{E030}
\begin{split}
\K_g = \big\{ (X,L): &\ X \text{ is a $K3$ surface, } L\in \Pic(X) \text{ is ample primitive}\\
 &\quad \text{ and } L^2 = 2g - 2 \big\}
\end{split}
\end{equation}
and $\CS_g = \{ (X, L, p): (X,L)\in \K_g, p\in X \}$.

Let $\CC_{g,n}\subset \CS_g$ be a closed subscheme of $\CS_g$
whose fiber over a general point $(X, L)\in \K_g$ is the union of all irreducible
rational curves in the linear series $|nL| = \BP H^0(X, nL)$. Our main theorem is

\begin{thm}\label{THM001}
For all $g\ge 2$, the set
\begin{equation}\label{E005}
\bigcup_{n=1}^\infty \CC_{g,n}
\end{equation}
is dense in $\CS_g$.
\end{thm}

Using an elementary topological argument, we can easily conclude the following (actually equivalent) statement.

\begin{cor}\label{COR001}
For all $g \ge 2$, the set
\begin{equation}\label{E006}
\left\{(X,L)\in \K_g: \bigcup_{n=1}^\infty \CC_{X, nL} \text{ is not dense in } X\right\}
\end{equation}
is {of the first Baire category}, i.e., a countable union of nowhere dense subsets in $\K_g$ under the usual topology,
where $\CC_{X,nL}$ is the fiber of $\CC_{g,n}$ over $(X, L)$.
Hence the set of $K3$ surfaces of genus $g$ whose rational curves are dense is {of the second Baire category}.
\end{cor}

This partially answers a question raised in \cite{C-L} (Conjecture 1.2),
although we expect that the union of rational curves
are dense on every projective $K3$ surface, not only the general ones.
However, the method here does not lend itself to handle every projective $K3$ surface.
On the other hand, it is unknown whether the union of rational curves is dense in
the Zariski topology on every projective $K3$ surface.

\begin{conj}\label{CONJ001}
The union of rational curves is dense in the Zariski topology on every projective $K3$ surface $X$.
That is, there are infinitely many rational curves on $X$.
\end{conj}

\begin{rem}\label{REM001}
This was known for a very general $K3$ surface using a deformational argument \cite{M-M}.
However, to deal with every projective $K3$ surface, some new methods are needed. Recently, 
some substantial progress has been made on the conjecture.
This was proved in \cite{BHT} for $g = 2$ and $\Pic(X) = \BZ$ using characteristic $p$ reduction.
Their method was further developed in \cite{L-L}, where the conjecture was settled in all major cases
with the only exception $\rank \Pic(X) = 2$ {(see further remark below)}. Also the conjecture was known for all elliptic $K3$ surfaces
\cite{BT1} (see also \cite{H-T}).
So the only unknown cases are $K3$ surfaces of Picard rank two which do not admit an elliptic fibration.
However, their method does not seem to apply to the strong topology.
Density of rational curves on $K3$ surfaces in both Zariski and strong topologies is related to
Lang's conjecture on these surfaces \cite{La}.
\end{rem}

\begin{rem}\label{REM100}
It was pointed out to us by the referee that Zariski density of rational curves is known for ``most'' $K3$ surfaces of Picard rank
$2$ since such a surface $X$ either has an infinite automorphism group or admits an elliptic fibration if $X$ does not contain a
$(-2)$-curve. So the only outstanding cases of \conjref{CONJ001} are $K3$ surfaces of Picard rank $2$ containing $(-2)$-curves, e.g.,
a quartic surface with a node.
\end{rem}

Although we are unable to prove the density of rational curves on every $K3$ surface in the strong topology,
we can do this for an elliptic $K3$ surface $X$ as long as the elliptic fibration $X\to \PP^1$ admits a rational
non torsion (nt) multisection. Here a rational nt multisection $C$ of $\pi:X\to\PP^1$ is an irreducible rational curve $C\subset X$
such that $C$ meets the general fiber $X_b$ of $\pi$ at (at least) one point $p$ satisfying $L - mp\not\in
J(X_b)_\tors$, where $L$ is an ample line bundle on $X$, $m = L\cdot X_b$ and $J(X_b) = \Pic^0(X_b)$ is the
Jacobian of the elliptic curve $X_b$ \cite{BT1}.

Using the classical Kronecker's theorem (see \ssecref{SS001})
together with a study of normal functions associated to an elliptic fibration, we are able to prove that

\begin{thm}\label{THM004}
The union of rational curves is dense in the strong topology on an elliptic $K3$ surface $\pi: X\to \PP^1$
if there exists a rational nt multisection $C$ of $\pi$.
\end{thm}

\begin{rem}\label{REM101}
The notion of rational nt multisections is also crucial in the work of Bogomolov-Tschinkel
\cite{BT1} and Hassett-Tschinkel \cite{H}.
Originally, we quoted a result in the
previously mentioned paper \cite[Theorem 1.8]{BT1} that every elliptic surface $X$ with
$\rank \Pic(X) \le 19$ has a rational nt multisection. It was again pointed out to us by the referee that
one of the constituent lemmas in that paper \cite[Lemma 3.26]{BT1} is incorrect, unfortunately (see \cite{H}
for a counterexample). At the moment, the above statement is still unknown and we cannot yet conclude the density of rational curves
on every elliptic $K3$ surface of Picard rank $\le 19$.
\end{rem}

We want to point out that density of rational curves on elliptic $K3$ surfaces does not imply the same
on a general $K3$ surface directly despite the fact that elliptic $K3$ surfaces are dense in the moduli space of $K3$ surfaces
since a rational {nt} multisection does not deform to a rational curve on a general $K3$ unless it is
a multiple of the polarization divisor.

\subsection{Density of elliptic curves}

For convenience, we will call a point {\it Baire general\/} if it lies in the complement of a countable union of nowhere dense
subsets.

Of course, every $K3$ surface $X$ is covered by one-parameter families of elliptic curves. It is natural to ask whether these curves
are dense when lifted to the first jet space $\BP T_X$ of $X$. Here the lifting $df: C\dashrightarrow \BP T_X$ of a map $f: C\to X$ is induced by
the map $f_*: T_C \to f^* T_X$ on the tangent sheaves.

For every $n\in \BZ^+$, we let $\W_{g,n}$ be the closure of the subscheme of $\BP H^0(\CS_g, nL)$ whose fiber over a general $(X, L)$ consists of irreducible elliptic curves in $|nL|$ and let
\begin{equation}\label{E003}
\E_{g,n} = \{ (X,L,E,p): (X,L,E)\in \W_{g,n}, p\in E\}
\subset \W_{g,n}\times_{\K_g} \CS_g
\end{equation}
be the universal family over $\W_{g,n}$.

\begin{thm}\label{THM002}
Let $\varphi: \E_{g,n} \dashrightarrow \BP T_{\CS_g/\K_g}$ be the rational map 
induced by the map
\begin{equation}\label{E004}
T_{\E_{g,n}/\W_{g,n}} \to T_{\CS_g/\K_g}
\end{equation}
on the relative tangent sheaves. Then
\begin{equation}\label{E000}
\bigcup_{n=1}^\infty \varphi(\E_{g,n})
\end{equation}
is dense in $\BP T_{\CS_g/\K_g}$ for all $g\ge 2$,
where $\varphi(\E_{g,n})$ is the proper transform of $\E_{g,n}$ under $\varphi$.
\end{thm}

It follows that the union of $\varphi(\E_{X,nL})$ is dense in $\BP T_X$ for a Baire general $(X, L)\in \K_g$,
where $\E_{X,nL}$ is the fiber of $\E_{g,n}$ over the point $(X, L)\in \K_g$.

\subsection{Hyperbolic geometry of $K3$ surfaces}

One of the reasons we are interested in the elliptic curves on a $K3$ surface $X$
comes from the fact that they are the images of holomorphic maps $\BC\to X$. So they
are closely related to the hyperbolic geometry of $X$. {Let us recall} 
the definition of  the {\it Kobayashi-Royden\/} (KR) pseudo-metric on a complex
manifold $X$ (cf. \cite{K}): for a point $p\in X$ and a nonzero tangent vector
$v\in T_{X, p}$, we define
\begin{equation}\label{E033}
\begin{split}
||v||_\kappa = \inf \{ \lambda > 0: & \exists \text{ a holomorphic map } f:
  \Delta\to X\\
&\text{ with } f(0) = p, f_*(\partial/\partial z) =
  \lambda^{-1} v \}
\end{split}
\end{equation}
Obviously, if there is a holomorphic $f: \BC\to X$ such that $f(0) = p$ and
$f_* (\partial/\partial z)
= v$ for some tangent vector $v\in T_{X,p}$, then $||v||_\kappa = 0$. In particular, if there is
holomorphic dominant map $f: \BC^n\to X$, then the KR pseudo-metric vanishes
everywhere on $X$. In \cite{Bu-L}, 
G. Buzzard and S. Y. Lu classified all the algebraic surfaces that are holomorphically
dominable by $\BC^2$.
They settled every single case except $K3$ surfaces,
for which they proved all elliptic and Kummer $K3$ surfaces can be holomorphically dominated by $\BC^2$.
But it is unknown whether a general $K3$ surface can be dominated by $\BC^2$ or has everywhere
vanishing KR pseudo-metric, although this is expected to be true.

\begin{conj}[Buzzard-Lu]\label{CONJ002}
Every complex $K3$ surface is holomorphically dominable by $\BC^2$. As a consequence, it has everywhere
vanishing KR pseudo-metric.
\end{conj}

By \thmref{THM002}, we know at least that the following holds.

\begin{cor}\label{COR002}
For $g\ge 2$, a Baire general $(X,L)\in \K_g$ and a Baire general $p\in X$, the set
$\{ v\in T_{X,p}: ||v||_\kappa = 0 \}$
is dense in $T_{X,p}$.
\end{cor}


The layout of this paper is as follows. We will prove our main theorems in \secref{SEC002}. In \secref{SEC003},
we will re-interpret a key step of our proof in terms of Poincar\'e normal functions and prove \thmref{THM004}.

\medskip
We are grateful to the referee for suggesting improvements to our paper.

\section{Proofs of \thmref{THM001} and \ref{THM002}}\label{SEC002}

\subsection{Elliptic $K3$ surfaces}

Our strategy is to show that rational curves are dense on $X$ for $(X,L)$ in a dense subset
of $\K_g$. Then \thmref{THM001} will follow easily.
It is well known that Kummer surfaces are dense in the moduli space of polarized $K3$ surfaces.
This implies that polarized elliptic $K3$ surfaces are dense.
A general projective elliptic $K3$ surface has Picard lattice given by
\begin{equation}\label{E001}
\begin{bmatrix}
2g-2 & m\\
m & 0
\end{bmatrix}
\end{equation}
where $m$ is a positive integer.
That is, the Picard group $\Pic(X)$ of $X$ is generated by effective classes $L$ and $F$ satisfying
\begin{equation}\label{E002}
L^2 = 2g - 2, LF = m \text{ and } F^2 = 0
\end{equation}
and the elliptic fibration $\pi: X\to \BP^1$ is given by the pencil $|F|$.

Let $\CP_{g,m}$ be the moduli space of the triples $(X, L, F)$, where $X$ is a $K3$ surfaces whose Picard lattice contains
\eqref{E001}, generated by $L$ and $F$, as a primitive sublattice.
Slightly abusing terminology, we sometimes treat $\CP_{g,m}$ as
a subscheme of $\K_g$; more precisely, it is a finite cover of a subscheme of $\K_g$.

The general theory of $K3$ surfaces tells us that $\CP_{g,m}$
is irreducible of codimension $1$ in $\K_g$ for each
pair $(g,m)$. Also the union of $\CP_{g,m}$ is dense in $\K_g$, as mentioned above. In fact, it follows from the standard theory
on the periodic domain of $K3$ surfaces (cf. \cite{BPV}) that we have the slightly stronger statement:
 \begin{equation}\label{E007}
\bigcup_{2|m} \CP_{g,m} \text{ is dense in $\K_g$}
\end{equation}
for all $g \ge 2$.

\begin{rem}\label{REM002}
The choice of $m$ being even is purely technical. As we will see,
it simplifies the construction of the degeneration of elliptic $K3$ surfaces.
It could be removed at the cost of making our later argument more complicated.
\end{rem}

\subsection{Dynamics under self rational maps}\label{SS001}

An elliptic $K3$ surface admits self rational maps induced by fiberwise 
elliptic curve endomorphism (cf. \cite{D}).

Let $(X, L)\in \CP_{g,m}$. Fixing $A\in \Pic(X)$ with $A F = a$,
we can construct a rational map
$\phi_A : X\dashrightarrow X$ by sending a point $p$ lying on a smooth fiber $X_q = \pi^{-1}(q)$ to
the point $A - (a-1) p$ on $X_q$ using the group structure of the elliptic curve $X_q$,
by which we mean that we send $p$ to the unique point $p'\in X_q$ given by
\begin{equation}\label{E300}
A \big|_{X_q} \sim_{\rat} (a-1)p + p'
\end{equation}
on $X_q$. Obviously,
$\phi_A$ is dominant unless $a = 1$. Of course, this construction works for all
fibrations of abelian varieties, not just elliptic $K3$'s. 

Let $C\subset X$ be an irreducible rational curve which is not contained in a fiber of $\pi$.
The proper transform $\phi_A(C)$ of $C$ under $\phi_A$ is also an irreducible rational curve on
$X$ not contained in a fiber. Naturally, we expect the following to be true.

\begin{prop}\label{PROP001}
For all $g,m\in \BZ^+$ satisfying $g\ge 2$ and $2|m$ and
a Baire general $(X, L)\in \CP_{g,m}$, there exists an irreducible rational curve
$C\subset X$ such that the set
\begin{equation}\label{E013}
\bigcup_{A\in \Pic(X)} \phi_A(C) = \bigcup_{n\in \BZ} \phi_{nL}(C)
\end{equation}
is dense on $X$.
\end{prop}

We cannot yet conclude \thmref{THM001} from \propref{PROP001} since $\phi_A(C)$ may not lie on
the fiber $\CC_{g,n}$ over the point $(X,L)$. Indeed, if $C\in |a L + bF|$,
$\phi_{kL}(C)\in |aL + b_k F|$ for some $b_k\in \BZ$.
As $|k|\to\infty$, $b_k\to\infty$ since we have only finitely many
rational curves in each linear series. Hence
$\phi_{kL}(C)\not\sim_\text{rat} nL$ for all $n\in \BZ$,
when $|k|$ is sufficiently large. So the rational curve $\phi_{kL}(C)$ alone cannot be deformed 
to a rational curve on a general $K3$ surface.
But we can find a rational curve $B_k\subset X$
such that $B_k + \phi_{kL}(C)\sim_\text{rat} nL$ for some $n\in \BZ$ and
the union $B_k \cup \phi_{kL}(C)$ can be deformed to an irreducible rational curve on a
general $K3$ surface. Namely, we can prove the following.

\begin{prop}\label{PROP002}
For all $g,m\in \BZ^+$ satisfying $g\ge 2$ and $2|m$, 
a general $(X,L)\in \CP_{g,m}$ and an irreducible rational curve $C\subset X$ such that
$C\not\sim_\text{rat} lL$ for any $l\in \BZ$,
\begin{itemize}
\item there exists an irreducible rational curve $B\subset X$ such that
$B\cup C$ lies on an irreducible component of $\CC_{g,n}$ that dominates $\K_g$;
\item there exists an irreducible elliptic curve $B\subset X$
such that $B\cup C$ lies on an irreducible component of $\E_{g,n}$ that
dominates $\CS_g$.
\end{itemize}
\end{prop}

Clearly, \propref{PROP001} and \ref{PROP002} together
will give us \thmref{THM001} and \ref{THM002}.

Let $X_q$ be a general fiber $\pi$ and $p\in X_q\cap C$. Then $\phi_{nL}$ sends $p$ to 
the point
\begin{equation}\label{E014}
\phi_{nL}(p) = nL - (mn-1)p
\end{equation}
and hence
\begin{equation}\label{E120}
\phi_{nL}(p) - p = n (\phi_L(p) - p) = n(L - mp)
\end{equation}
in the Jacobian $\Pic^0(X_q) = J(X_q)$ of the elliptic curve $X_q$.

\propref{PROP001} will follow if we can prove that the subgroup of $J(X_q)$
generated by $L-mp$ is dense.
So we naturally ask which points on an elliptic curve, or more generally a compact complex torus, generate a dense
subgroup.
{This is answered by the classical Kronecker's theorem (cf. \cite[Chap. XXIII]{H-W}). For convenience, we
put it in the following form.} 

{\begin{thm}[Kronecker's Theorem]\label{THM200}
Let $A = \BR^n /\BZ^n$ be a compact real torus of dimension $n$. 
For a point $p = (x_1, x_2, ..., x_n)\in A$, $\BZ p = \{ kp: k\in \BZ\}$ is dense in $A$ if and only
if $1, x_1, x_2, ..., x_n$ are linearly independent on $\BQ$.
In particular, the set
\begin{equation}\label{E034}
\big\{p\in A:  \BZ p \text{ is not dense in } A\big\}
\end{equation}
is of the first Baire category.
\end{thm}}

There are two ways we can show that $L-mp$ generates a dense subgroup of $J(X_q)$
using {Kronecker's theorem}. One way is via normal functions. This will be done in \secref{SEC003}.
The other way is to show that $L - m p$ is general in $J(X_q)$ as $X$ and $q$ vary.

First of all, we have to make what we mean by ``general in $J(X_q)$'' precise. Let 
\begin{equation}\label{E015}
\CS_{g,m} = \CS_g\times_{\K_g} \CP_{g,m}
\end{equation}
be the pullback of the universal family $\CS_g$ to $\CP_{g,m}\subset \K_g$ and
let $\CC_{g,m,A}\subset \CS_{g,m}$ be the closed subscheme
whose fiber over a general point $(X,L)\in \CP_{g,m}$ is the union of all irreducible rational curves in $|A|$, where
$A\in \Pic(\CS_{g,m}/\CP_{g,m})$. Note that we have an elliptic fibration
\begin{equation}\label{E044}
\pi: \CS_{g,m} \to \BP^1\times \CP_{g,m}
\end{equation}
given by the pencil $|F|$. The induced map $\CC_{g,m,A}\to \BP^1\times\CP_{g,m}$ is generically finite if it is dominant.

Let $p\in \CC_{g,m,A}$ be a point over a general point $q = \pi(p)\in \BP^1\times\CP_{g,m}$. We have a map $p\to J(E)$ by sending
$p$ to $L - mp$, where $E = \pi^{-1}(q)\subset \CS_{g,m}$
is the fiber of $\pi$ over $q$. Note that $J(E) = \Pic^0(E)$
is the elliptic curve $E$ with a base point $0$ corresponding to the trivial bundle $\CO_E$.
So we have two marked points $(0, L - mp)$ on $J(E)$. Namely, we have a well-defined map
\begin{equation}\label{E045}
\gamma: \CC_{g,m,A} \to \overline{\m}_{1,2}
\end{equation}
sending $p$ to $(J(E), 0, L - mp)$, where $\overline{\m}_{g,n}$
is the moduli space of stable curves of genus $g$ with $n$ marked
points. By saying $L-mp$ is general, we simply mean that $\gamma$ is dominant.

\begin{lem}\label{LEM003}
For all $g,m\in\BZ^+$ satisfying $g\ge 2$ and $2|m$,
there exists an irreducible component of $\CC_{g,m,A}$ dominating $\BP^1\times\CP_{g,m}$ via $\pi$ and
dominating $\overline{\m}_{1,2}$ via $\gamma$ for some $A\in \Pic(\CS_{g,m}/\CP_{g,m})$.
\end{lem}

This, together with Kronecker's theorem, will give us \propref{PROP001}.

If $\CC_{g,m,A}$ dominates $\BP^1\times\CP_{g,m}$, it
obviously dominates $\overline{\m}_{1,1}$ by
\begin{equation}\label{E052}
\CC_{g,m,A}\xrightarrow{\gamma} \overline{\m}_{1,2}\xrightarrow{\tau} \overline{\m}_{1,1}
\end{equation}
where $\tau$ is the forgetting map. So to show that $\gamma$ is dominant, it suffices
to show that the closure of the image of $\gamma$ contains the boundary component
$\overline{\m}_{0,4}\subset \overline{\m}_{1,2}$. The proof of this fact relies on a degeneration
argument.

\subsection{Deformation of $K3$ surfaces}\label{SSECDEFORM}

Following the idea in \cite{CLM}, we can deform a $K3$ surface to a union of two rational surfaces.
Let $R = R_1\cup R_2$ be the union of
two smooth rational surfaces $R_1$ and $R_2$ meeting transversely along a smooth elliptic
curve $D = R_1\cap R_2$ where $D = -K_{R_i}$ in $\Pic(R_i)$ for $i=1,2$. We see that 
$R$ is simply connected and the dualizing sheaf $\omega_R$ of $R$ is trivial. So it is expected
that $R$ can be deformed to a $K3$ surface. The deformation of $R$ is governed by the map
\begin{equation}\label{E040}
\begin{split}
\Ext(\Omega_R, \CO_R) \to H^0(T^1(R)) &= H^0({\mathcal E}xt(\Omega_R, \CO_R))\\
& = H^0(\CO_D(-K_{R_1} - K_{R_2})).
\end{split}
\end{equation}
Then $R$ can be deformed to a $K3$ surface only if the image of the above map is base point free
in $H^0(T^1(R))$. That is, $R_\text{sing} = D$ can be smoothed when $R$ deforms. This puts
some restrictions on $R_i$. A necessary condition is that $\CO_D(-K_{R_1} - K_{R_2})$ is base point
free. It can be guaranteed if we choose $R_i$ to be Fano.

A deformation of $R$ is a complex $K3$ surface, not necessarily projective.
In order to deform $R$ to a projective $K3$ surface, in particular, to deform $R$ to a $K3$ surface in $\CP_{g,m}$, we need
to construct $R$ in such a way that it has
two line bundles $L$ and $F$ satisfying \eqref{E002}. Let $L_i = L|_{R_i}$ and $F_i = F|_{R_i}$ for $i=1,2$. Then
\begin{equation}\label{E050}
L_1\bigg|_D = L_2\bigg|_D \text{ and } F_1\bigg|_D = F_2\bigg|_D.
\end{equation}
Indeed, $R$ is constructed by gluing $R_1$ and $R_2$ transversely along $D$ such that
\begin{equation}\label{E051}
e_1^* L_1 = e_2^* L_2 \text{ and } e_1^* F_1 = e_2^* F_2
\end{equation}
where $e_i$ is the embedding $D\hookrightarrow R_i$ for $i=1,2$.

As in \cite{CLM} and \cite{C},
we can degenerate every $K3$ surface of genus $g\ge 3$ to a union $R = R_1\cup R_2$ as follows: 
\begin{itemize}
\item if $g\ge 3$ is odd, we let $R_i \isom \BF_0 = \BP^1\times \BP^1$ and $R_1\cup R_2$ be polarized
by the ample line bundle $L$ where
\begin{equation}\label{E900}
L_i = L\bigg|_{R_i} = M_i + \frac{g-1}{2} G_i
\end{equation}
with $M_i$ and $G_i$ being the generators of $\Pic(R_i)$ satisfying
\begin{equation}\label{E150}
M_i^2 = G_i^2 = 0 \text{ and } M_iG_i = 1
\end{equation}
for $i=1,2$;
\item if $g\ge 4$ is even, we let $R_i\isom\BF_1 = \BP(\CO_{\BP^1} \oplus \CO_{\BP^1}(-1))$ and
$R_1\cup R_2$ be polarized by the ample line bundle $L$ where
\begin{equation}\label{E440}
L_i = L\bigg|_{R_i} = M_i + \frac{g}{2} G_i
\end{equation}
with $M_i$ and $G_i$ being
the generators of $\Pic(R_i)$ satisfying
\begin{equation}\label{E151}
-M_i^2 = M_iG_i = 1 \text{ and } G_i^2 = 0
\end{equation}
for $i=1,2$.
\end{itemize}

Note that this does not cover the case $g = 2$. The genus $2$ case will be treated separately in
\ssecref{SSECG2}.

Such $R$ can be deformed to a general $K3$ surface in $\K_g$. In order to deform it to a $K3$
surface in $\CP_{g,m}$, we need to have another line bundle $F\in \Pic(R)$ besides $L\in \Pic(R)$.
Here we simply let
\begin{equation}\label{E152}
F_i = F\bigg|_{R_i} = \frac{m}{2} G_i
\end{equation} 
with the ``tricky'' requirement that
\begin{equation}\label{E153}
\CO_D(G_1 - G_2)\in \Pic^0(D) = J(D) \text{ is $(m/2)$-torsion.}
\end{equation}
We glue $R_1$ and $R_2$ in such a way that \eqref{E050} are the only relations
between $\Pic(R_i)$, with $L_i$ and $F_i$ given by 
\eqref{E900}, \eqref{E440} and \eqref{E152}, respectively.
More precisely, the kernel $\Pic(R)$ of the map
\begin{equation}\label{E053}
\Pic(R_1)\oplus \Pic(R_2)\xrightarrow{e_1^* - e_2^*} \Pic(D) 
\end{equation}
is freely generated by $L = L_1 \oplus L_2$ and $F = F_1 \oplus F_2$. Numerically, we have
\begin{equation}\label{E054}
L_i^2 = g - 1, L_i F_i = \frac{m}{2} \text{ and } F_i^2 = 0.
\end{equation}

\begin{rem}\label{REM003}
It may appear that $F$ is not primitive by \eqref{E152}. It actually is
since $G_1 - G_2$ is a torsion point of $J(D)$ of order $m/2$ and hence there does not exist $k\in \BZ$ such
$\CO_D(kG_1) = \CO_D(kG_2)$ unless $(m/2)|k$. It may also appear that $h^0(F) = m/2+1$ by
\eqref{E152}. Actually, $h^0(F) = 2$, i.e., $|F|$ is a pencil, again by \eqref{E153}.
Indeed, a member of $|F|$ is a union $N_1\cup N_2\cup ... \cup N_m$ where
\begin{itemize}
\item
$N_k\subset R_1$ and $N_k\in |G_1|$ for $k$ odd and
$N_k\subset R_2$ and $N_k\in |G_2|$ for $k$ even;
\item $\cup N_k$ meets $D$ at points $q_1, q_2, ..., q_m$ such that
\begin{equation}\label{E055}
N_k\cdot D = q_k + q_{k+1}
\end{equation}
for $1\le k\le m$, where we let $q_{m+1} = q_1$.
\end{itemize}
Obviously, such a union $\cup N_k$ moves in a base point free pencil.
\end{rem}

Such $R$ can be deformed to $K3$ surfaces in $\CP_{g,m}$. That is, there exists a one-parameter
{family} $\CS$ over the disk $\Delta = \{|t| < 1\}$ and
two line bundles $L$ and $F\in \Pic(\CS/\Delta)$ such that
$(\CS_t, L)\in \CP_{g,m}$ for $t\ne 0$ and $\CS_0 = R$ is the union $R$ with $L$
and $F$ constructed as above. 
The proofs of \lemref{LEM003} and \propref{PROP002} both depend on the construction of certain
rational curves on the general fibers $\CS_t$. Our strategy is
to produce a reducible rational curve on the central fiber $R$, called a limiting rational curve
in \cite{C}, and show that it can be
deformed to an irreducible rational curve on the general fibers.

\begin{lem}\label{LEM004}
For all $g,m\in \BZ^+$ satisfying $g\ge 2$ and $2|m$ and
a general $(X,L)\in \CP_{g,m}$, there is an irreducible nodal rational curve in $|aL + bF|$
for all $a\in \BZ^+$ and $b\in \BZ$ satisfying
\begin{equation}\label{E063}
\max\left(2a \left\lfloor \frac{g-1}{2}\right\rfloor, a\right) + bm > 0
\end{equation}
and $b^2 + (g-2)^2 \ne 0$.
\end{lem}

\begin{proof}[Proof (when $g\ge 3$)]
Our construction of limiting rational curves $C_1\cup C_2$ with $C_i\subset R_i$ is very similar
to the construction in \cite{C}, but with some added difficulties. Namely, we have to make sure
that 
\begin{equation}\label{E069}
\begin{split}
&\text{there does not exist $C_1'\cup C_2' \subsetneq C_1\cup C_2$}\\
&\quad \text{such that $C_1'\cup C_2'\in
|a'C + b'F|$ for some $a', b'\in \BZ$;}
\end{split}
\end{equation}
otherwise, a deformation of $C_1\cup C_2$
onto a general fiber $\CS_t$ is not necessarily irreducible. This is a little trickier here
due to the fact $\rank\Pic(R) = 2$ and the condition \eqref{E153}.

The one-parameter family $\CS$ has sixteen rational double points $p_1, p_2, ..., p_{16}$ lying on
$D$, which are precisely the zeros of a section in $H^0(T^1(R))$ that is in turn the image of
the Kodaira-Spencer class of $\CS/\Delta$ under the map \eqref{E040}. So these sixteen points
satisfy
\begin{equation}\label{E056}
\begin{split}
\CO_D(p_1+p_2+...+p_{16}) &= \CO_D(-K_{R_1} - K_{R_2})\\
& =
\begin{cases}
\CO_D(2M_1 + 2G_1 + 2M_2 + 2G_2) & \text{if } 2\nmid g\\
\CO_D(2M_1 + 3G_1 + 2M_2 + 3G_2) & \text{if } 2\mid g
\end{cases}
\end{split}
\end{equation}
and this is the only relation among $p_1, p_2, ..., p_{16}$ for a general choice of $\CS$.

We write
\begin{equation}\label{E057}
(aL + bF)\bigg|_{R_i} = a M_i + \left(a\left\lfloor \frac{g}{2}\right\rfloor + \frac{bm}{2}\right) G_i =
a M_i + l G_i.
\end{equation}

\medskip

\noindent{\bf Case $2\nmid g$ and $a \le l$.}
We let
\begin{equation}\label{E058}
C_i = I_{i1}\cup I_{i2}\cup ... \cup I_{i,a-1}\cup J_{i1}\cup J_{i2}\cup ...\cup J_{i,a-1}
\cup \Gamma_i
\end{equation}
be the curve on $R_i$ ($i=1,2$) with irreducible components
$I_{ij}\in |G_i|$, $J_{ij}\in |M_i|$ and $\Gamma_i\in |M_i + (l-a+1) G_i|$ given by
\begin{equation}\label{E059}
\begin{split}
I_{11} \cdot D &= p_1 + q_1,\ J_{21}\cdot D = q_1 + q_2\\
I_{12}\cdot D &= q_2 + q_3,\ J_{22}\cdot D = q_3 + q_4\\
&\ldots\\
I_{1,a-1}\cdot D &= q_{2a-4} + q_{2a-3},\ J_{2,a-1}\cdot D = q_{2a-3} + q_{2a-2}
\end{split}
\end{equation}
\begin{equation}\label{E060}
\begin{split}
I_{21} \cdot D &= p_2 + r_1,\ J_{11}\cdot D = r_1 + r_2\\
I_{22}\cdot D &= r_2 + r_3,\ J_{12}\cdot D = r_3 + r_4\\
&\ldots\\
I_{2,a-1}\cdot D &= r_{2a-4} + r_{2a-3},\ J_{1,a-1}\cdot D = r_{2a-3} + r_{2a-2}
\end{split}
\end{equation}
and
\begin{equation}\label{E061}
\Gamma_1\cdot D = p_2 + q_{2a -2} + (2l-2a + 2) s, \ \Gamma_2\cdot D = p_1 + r_{2a -2} + (2l-2a + 2) s
\end{equation}
where $q_j, r_j$ and $s$ are points on $D$ and we let $q_0 = p_1$ and $r_0 = p_2$. Intuitively, $C_1\cup C_2$ is the union of
two chains of curves, one starting at $p_1$ and the other starting at $p_2$,
consisting of curves in $|G_i|$ and $|M_i|$ alternatively and finally ``joined'' by $\Gamma_1$
and $\Gamma_2$. We see that \eqref{E069} holds because
$p_1$ and $p_2$ are two general points on $D$ and $G_i$ and $M_{3-i}$ are linearly independent in $\Pic_\BQ(D)$ for each $i=1,2$.

\medskip

\noindent{\bf Case $2\nmid g$ and $a > l$.}
Note that \eqref{E063} implies $l > 0$ when $g$ is odd.
We use the same construction as above for $a \le l$ by
simply switching $G_i$ and $M_i$ and switching $a$ and $l$.

\medskip

\noindent{\bf Case $2|g$.}
Note that $l > a$ by \eqref{E063} when $g$ is even. Let
\begin{equation}\label{E062}
\alpha = \left\lfloor \frac{a}{2}\right\rfloor \text{ and }
\beta = \left\lfloor \frac{a-1}{2}\right \rfloor 
\end{equation}
and let
\begin{equation}\label{E064}
C_i = I_{i1}\cup I_{i2}\cup ...\cup I_{i\alpha} \cup J_{i1}\cup J_{i2}\cup ...\cup J_{i\beta} \cup \Gamma_i
\end{equation}
be the curve on $R_i$ ($i=1,2$) with irreducible components $I_{ij}, J_{ik}\in |M_i+G_i|$ and $\Gamma_i\in |M_i+ (l-a+1) G_i|$
given by 
\begin{equation}\label{E065}
\begin{split}
I_{11}\cdot D &= p_1 + p_3 + q_1,\ I_{21} \cdot D = p_2 + q_1 + q_2\\
I_{12}\cdot D & = p_1 + q_2 + q_3,\ I_{22} \cdot D = p_2 + q_3 + q_4\\
& \ldots\\
I_{1\alpha}\cdot D &= p_1 + q_{2\alpha -2} + q_{2\alpha -1},\ I_{2\alpha} \cdot D = p_2 + q_{2\alpha-1} + q_{2\alpha},
\end{split}
\end{equation}
\begin{equation}\label{E066}
\begin{split}
J_{21}\cdot D &= p_1 + p_4 + r_1,\ J_{11} \cdot D = p_2 + r_1 + r_2\\
J_{22}\cdot D &= p_1 + r_2 + r_3,\ J_{12}\cdot D = p_2 + r_3 + r_4\\
& \ldots\\
J_{2\beta}\cdot D &= p_1 + r_{2\beta -2} + r_{2\beta -1},\ J_{1\beta} \cdot D = p_2 + r_{2\beta-1} + r_{2\beta},
\end{split}
\end{equation}
and
\begin{equation}\label{E067}
\begin{split}
\Gamma_1\cdot D &= p_4 + q_{2\alpha} + (\alpha - \beta) p_2 + (2l - 2a -\alpha + \beta + 1) s,\\
\Gamma_2\cdot D &= p_3 + r_{2\beta} + (\alpha - \beta) p_1 + (2l - 2a -\alpha + \beta + 1) s
\end{split}
\end{equation}
where $q_j$, $r_k$ and $s$ are points on $D$ and we let $q_0 = p_3$ and $r_0 = p_4$.
Since $p_1, p_2, p_3, p_4$ are in general position on $D$, it is not hard to see that
\eqref{E069} holds.

\medskip

The curve $C_1\cup C_2$ constructed above has the following properties in addition to
\eqref{E069}:
\begin{itemize}
\item every component of $C_i$ is a smooth rational curve
and $C_i$ has simple normal crossing outside of $D$;
\item if $C_i$ and $D$ meet at a point $q\not\in \{p_1,p_2, ...,p_{16}\}$, there is only one branch
of $C_i$ locally at $q$, i.e., $C_i$ is smooth at $q$;
\item if $C_i$ and $D$ meet at a point $q\in \{p_1,p_2, ...,p_{16}\}$, all local branches of $C_i$
at $q$ meet transversely with each other and also transversely with $D$.
\end{itemize}
Then by the argument in \cite{C}, more specifically, {by \cite[Theorem 2.1 and 2.2]{C}, we can
show that }
$C_1\cup C_2$ can be deformed to an irreducible nodal rational
curve on the general fibers of $\CS/\Delta$. More precisely, there exists a flat family of curves
$\CC\subset \CS$, after a base change, such that $\CC_0 = C_1\cup C_2$ and $\CC_t$ is an irreducible
rational curve with only ordinary double points as singularities for $t\ne 0$.
\end{proof}

\begin{rem}\label{REM004}
The condition \eqref{E063} is trivially satisfied when we take $a >> |b|$. Therefore, for every 
$C\not\sim_\text{rat} lL$, there is an irreducible nodal rational curve in
$|nL - C|$ for $n$ sufficiently large. This is what we need for \propref{PROP002}.
\end{rem}

Now we are ready to prove \lemref{LEM003}.

\begin{proof}[Proof of \lemref{LEM003} (when $g \ge 3$)]
By \lemref{LEM004}, there is an irreducible component of $\CC_{g,m,A}$ dominating
$\BP^1\times\CP_{g,m}$, by setting e.g. $A = 2L + F$.

Let $\CS/\Delta$ be the family of $K3$ surfaces constructed above. One may think of $\CS$ as the
pullback of $\CS_{g,m}$ under a map $\Delta^* \to \CP_{g,m}$.
Let $\CC\subset \CS$ be a family of rational curves constructed in the proof of \lemref{LEM004}
with $\CC_t\in |2L + F|$. One may think of $\CC$ as an irreducible component of the pullback of
$\CC_{g,m,2L+F}$ to $\CS$. Correspondingly, we pull back the map $\gamma$ to $\CC$, i.e.,
\begin{equation}\label{E070}
\gamma: \CC\to \overline{\m}_{1,2}
\end{equation}
sending $p\in \CC_t$ to $(J(E_p), 0, L - mp)$, where $E_p$ is the fiber of
the projection $\pi: \CS\to \BP^1\times \Delta$ over the point $\pi(p)$.

It is enough to prove that 
\begin{equation}\label{E071}
\dim(\gamma(\CC)\cap \overline{\m}_{0,4}) = 1
\end{equation}
where we think of $\overline{\m}_{0,4}$ as a component of
$\overline{\m}_{1,2} \backslash \m_{1,2}$. Instead of directly studying $\gamma$, which roughly
maps $p$ to $L - mp$, we look at the map sending $p$ to $p - p'$, where $p'\ne p$
is another point of intersection between $E_p$ and $\CC_t$.

More precisely, we let $\T$ be the product $\CC\times_{\BP^1\times\Delta} \CC$ with diagonal removed.
We have a well-defined map
\begin{equation}\label{E072}
\xi: \T\to \overline{\m}_{1,2}
\end{equation}
sending $(p,p')\in \T$ to $(J(E_p), 0, p - p')$. 
Clearly, $\gamma$ is dominant if $\xi$ is; $\xi$ is dominant if
\begin{equation}\label{E073}
\dim(\xi(\T)\cap \overline{\m}_{0,4}) = 1.
\end{equation}

As $t\to 0$,
the fibers $E_p\in |F|$ of $\pi: \CS_t\to \BP^1$ will degenerate to a curve $N\in |F|$ on the central fiber
$\CS_0 = R_1\cup R_2$ as described in \remref{REM003}. That is, $N$ is a union
$N_1\cup N_2\cup ... \cup N_m$ given by \eqref{E055}.
For $N$ a general member of the pencil $|F|$, $N_1$ meets $C_1$ transversely at two distinct points
$p\ne p'\not\in D$, where $\CC_0 = C_1\cup C_2$ is the limiting rational curve constructed in the
proof of \lemref{LEM004}. Clearly, $(p, p')\in \T$. It is not hard to see that $\xi$ simply sends
$(p, p')$ to $(p, p', q, q')\in \overline{\m}_{0,4}$ as four points on $N_1\isom \BP^1$, where
$N_1\cap D = \{q, q'\}$. To show \eqref{E073}, it suffices to show that the moduli of
$(p, p', q, q')$ varies when $N$ moves in the pencil $|F|$. From the construction of
$C_1\cup C_2$, we see that $C_1$ has a node $r\not\in D$. For $N$ a general member of $|F|$,
$(p, p', q, q')$ are four distinct points on $N_1$. When $N_1$ passes through $r$, we have
$p = p' = r$ while $r\ne q$ and $r\ne q'$. So the moduli of $(p, p', q, q')$ changes as $N$ varies.
We are done.
\end{proof}

\subsection{Proof of \propref{PROP002}}

By \lemref{LEM004}, we can find an irreducible nodal rational curve $B\in |nL - C|$
for $n$ sufficiently large. It also follows that there is an irreducible nodal elliptic curve
$B\in |nL - C|$.
It remains to show that we can deform $B\cup C$ to a rational curve
if $g(B) = 0$ or an elliptic curve if $g(B) = 1$
on a general $K3$ surface.

Let us do the case $g(B) = 0$. 
We fix an intersection $p\in B\cap C$. Let $B^\nu$ and $C^\nu$
be the normalizations of $B$ and $C$, respectively, and let  
\begin{equation}\label{E080}
\eta: B^\nu\vee_p C^\nu \to B\cup C \subset X
\end{equation}
be a partial normalization of $B\cup C$, where $B^\nu\vee_p C^\nu$ is the union of $B^\nu$
and $C^\nu$ meeting transversely at a single point over $p$. We want to show that the stable
map $\eta: B^\nu\vee_p C^\nu \to X$ can be deformed when we deform $X$. This can be done by
computing the virtual dimension of the moduli space of stable maps to $K3$ surfaces. However,
it cannot be done in a naive way for the following well-known reason:
the virtual dimension of $\m_{nL,0}^X$ is
\begin{equation}\label{E081}
c_1(T_X)\cdot nL + \dim X - 3 = -1
\end{equation}
which is not the ``expected'' dimension $0$, where $\m_{\gamma,g}^X$ is the moduli space of
stable maps $\eta: A\to X$ of genus $g$ with $[\eta_* A ] = \gamma\in H_2(X,\BZ)$. 
The remedy for this situation is to replace $X$ by a
so-called ``twisted'' family of $K3$ surfaces, i.e., a complex deformation of $X$. This way
we have the right dimension and the stable maps $\eta\in \m_{nL,0}^X$ only deform onto
projective $K3$ surfaces.

Let $\CS/\Delta^2$ be a family of complex $K3$ surfaces over the 2-disk
$\Delta^2$ with $\CS_0 = X$ and the class $L\in H_2(\CS,\BZ)$. Then $\eta\in \m_{nL,0}^\CS$ and
the virtual dimension of $\m_{nL,0}^\CS$ is
\begin{equation}\label{E082}
c_1(T_\CS)\cdot nL + \dim \CS - 3 = 1.
\end{equation}
Let $V$ be an irreducible component of $\m_{nL,0}^\CS$ containing $\eta$.
Then $\dim V \ge 1$. Let
\begin{equation}\label{E083}
W = \{ t\in \Delta^2: L\in \Pic(\CS_t) \}
\end{equation}
be the subvariety of $\Delta^2$ parameterizing projective $K3$ surfaces polarized by $L$.
Obviously, $\dim W = 1$ and $\Pic(\CS_t) = \BZ$ for $t\ne 0\in W$ and a general choice of
$\CS$. Clearly, $V$ maps to $W$ under the projection $\CS\to \Delta^2$ and
it is obviously flat over $W$ since $\dim V_0 = 0$ and $\dim V \ge \dim W$.
Hence $V_t\ne\emptyset$ for $t\ne 0 \in W$.

\begin{rem}\label{REM006}
The above argument on the deformation of reduced stable maps to $K3$ surfaces is well known among the experts.
The first author learned it from Jun Li. Please also see \cite{B-L} and \cite[Sec. 2.2]{M-P}.
\end{rem}

The case $g(B) = 1$ follows from the same argument.
This finishes the proof of \propref{PROP002}
and hence \thmref{THM001} follows. We need to say a few things more for \thmref{THM002}.

By deformation theory, $B$ moves in a one-parameter family when $g(B) = 1$.
A general member of this family is an irreducible nodal elliptic curve
meeting $C$ transversely. In addition, the intersection $p\in B\cap C$ moves on $C$
when $B$ varies in the family. 
Now we let $\CS$ be a projective family of $K3$ surfaces 
polarized by $L$ over $\Delta$ with $\CS_0 = X$.
We can deform $B\cup C$ to an irreducible elliptic curve
on a general fiber $\CS_t$ of the family $\CS/\Delta$ by the argument above.
Namely, there exists a family of curves
$\CC\subset \CS$, after a base change, such that $\CC_0 = B\cup C$ and $g(\CC_t) = 1$ for $t\ne 0$.
Let $\nu: \CC^\nu\to \CC$ be the normalization of $\CC$. Then $\CC_0^\nu$ is the union
$B^\nu\vee_p C^\nu$ described above.
We lift $\nu: \CC^\nu \to \CS$ to $d\nu: \CC^\nu \dashrightarrow \BP T_{\CS/\Delta}$.
Let $\mu: \wt{\CC} \to \BP T_{\CS/\Delta}$ be the stable reduction of the map $d\nu$ and let
$\wt{B}$ and $\wt{C}\subset \wt{\CC}_0$ be the proper transforms of $B$ and $C$, respectively.

Since $B$ and $C$ meet transversely at $p$, the images of the tangent spaces $T_{B,p}$ and
$T_{C,p}$ in $T_{X,p}$ differ. Consequently, $\mu(\wt{B})$ and $\mu(\wt{C})$ meet
$\BP T_{X,p}$ at two distinct points, where $\BP T_{X,p}\isom \BP^1$
is the fiber of $\BP T_{\CS/\Delta}/\CS$ over the point $p\in \CS$. Therefore, $\wt{B}$ and $\wt{C}$
are disjoint on $\wt{\CC}_0$ and they must be joined by a tree of rational curves that dominates
$\BP T_{X,p}$. That is, $\BP T_{X,p} \subset \mu(\wt{C})$ and hence
$\BP T_{X,p}\subset \varphi(\E_{g,n})$. As $p$ moves on $C$, we see that
\begin{equation}\label{E100}
\bigcup_{p\in C} \BP T_{X,p} \subset \overline{\varphi(\E_{g,n})}.
\end{equation}
We take $C$ to be a member of a sequence of rational curves which are dense on $X$. Hence
\begin{equation}\label{E101}
\BP T_X \subset \bigcup_{n=1}^\infty \overline{\varphi(\E_{g,n})}
\end{equation}
and \thmref{THM002} follows.

\subsection{The case $g = 2$}\label{SSECG2}

A $K3$ surface in $\CP_{2,m}$ can still be degenerated to a union $R_1\cup R_2$ with $R_i\isom \BF_1$ and
$L_i$ and $F_i$ given by \eqref{E440}, \eqref{E152} and \eqref{E153}, just as in the case that $g\ge 4$
is even. Let $\CS/\Delta$ be the corresponding family of $K3$ surfaces with $\CS_0 = R_1\cup R_2$ and
$(\CS_t, L)\in \CP_{2,m}$. Such $\CS$ is projective over $\Delta$ since $L + nF$ is relatively ample
over $\Delta$ for all $n > 0$. However, $L$ is big and nef but not ample over $\Delta$ itself. Indeed,
the birational
map $\psi: \CS\to \QQ$ given by $|nL|$ for $n\ge 2$ contracts the two exceptional curves $M_i$.
The 3-fold $\QQ$ is a family of $K3$ surfaces in $\CP_{2,m}$ over $\Delta$ whose central fiber
$\QQ_0 = S_1\cup S_2$ is a union of $S_i\isom \PP^2$ meeting transversely along an elliptic curve
$D = S_1\cap S_2$. Here we use the same notation $D$ for both intersections $R_1\cap R_2$ and
$S_1\cap S_2$.

The two curves $M_i$ are contracted by $\psi$ to two rational double points $p_{17}$ and $p_{18}$
of $\QQ$ on $D = S_1\cap S_2$. Indeed, $\QQ$ has eighteen rational double points $p_1, p_2, ..., p_{16},
p_{17}, p_{18}$ on $D$ by deformation theory, where $p_1, p_2, ..., p_{16}$ are the images of the rational double points of $\CS$ under $\psi$. Again we use the same notations $p_1, p_2, ..., p_{16}$ for both
the rational double points of $\CS$ and their images under $\psi$.

One subtle point is that $M_i$ are contracted to the same point $p_{17} = p_{18}$ where $\QQ$ has a
singularity of the type $xy = tz^2$ when $m = 2$. Such a singularity can be analyzed in the same way as
rational double points. Basically, we have two rational double points ``collide'' in this special case.
However, we can save ourselves some trouble in dealing with this ``corner'' case by simply assuming 
that $m\ge 4$ since
\begin{equation}\label{E400}
\bigcup_{\genfrac{}{}{0pt}{}{2|m}{m \ge m_0}} \CP_{g,m}
\end{equation}
is obviously dense in $\K_g$ for all $m_0$.
For our purpose, we may simply assume $m$ to be sufficiently large. So $p_{17}$ and $p_{18}$ are two
distinct points on $D$ and $R_i$ is the blowup of $S_i$ at $p_{16+i}$ for $i=1,2$, respectively. And
$\psi: \CS\to \QQ$ is a small resolution of $\QQ$ at $p_{17}$ and $p_{18}$. It is well known that there
are flops of $\CS$ with respect to $M_i$. Namely, we have the diagram
\begin{equation}\label{E401}
\xymatrix{
\CS \ar@{-->}[rr]\ar[dr]_\psi & & \CS'\ar[dl]^{\psi'}\\
& \QQ
}
\end{equation}
where $\CS'$ is the 3-fold obtained from $\CS$ by flops with respect to $M_1$ and $M_2$. That is,
the central fiber $\CS_0' = R_1' \cup R_2'$ of $\CS'$ is a union of $R_i'\isom \BF_1$ with $R_i'$ the
blowup of $S_i$ at $p_{19-i}$ for $i = 1,2$.

Let $L'$ and $F'\in \Pic(\CS'/\Delta)$ be the proper transforms of $L$ and $F$, respectively, and let
$M_i'$ and $G_i'$ be the generators of $\Pic(R_i')$ given in the same way as \eqref{E151}. It is not hard
to see that
\begin{equation}\label{E402}
L_i' = L'\bigg|_{R_i'} = M_i' + G_i'\text{ and } F_i' = F'\bigg|_{R_i'} = m M_i' + \frac{m}{2} G_i'.
\end{equation}
So we can work with either $\CS$ or $\CS'$ to produce rational curves in $|a L + b F|$ on $\CS_t$ or equivalently
$|aL' + bF'|$ on $\CS_t'$, depending on the sign of $b$.

\begin{proof}[Proof of \lemref{LEM004} when $g = 2$]
When $b > 0$, we have
\begin{equation}\label{E403}
(a L + b F)\bigg|_{R_i} = a M_i + \left(a + \frac{bm}{2}\right) G_i
\end{equation}
with $a + bm/2 > a$. Hence we may use the same construction of limiting rational curves $C_1\cup C_2$ as in
the case of $g$ being even and $g\ge 4$.

When $b < 0$, we have
\begin{equation}\label{E404}
(a L' + b F')\bigg|_{R_i'} = (a + bm) M_i' + \left(a + \frac{bm}{2}\right) G_i'
\end{equation}
where $a + bm > 0$ by \eqref{E063} and $a + bm/2 > a + bm$. So we may use the same construction again by working with
$\CS'$.
\end{proof}

The proof of \lemref{LEM003} goes through without any change since we are using the limiting rational curves
in $|2L+F|$ for which no flops are needed.

\section{Normal Functions Associated to Elliptic Fibrations}\label{SEC003}

Here we will give another proof of \propref{PROP001} via the theory of normal functions.
Roughly, we will show that if $L - mp$ fails to generate a dense subgroup of $J(X_q)$ for a general point
$q\in \PP^1$, then it has to be torsion for all $q$. The advantage of this approach is that it does not
seem to depend on the general moduli of $X$, although we do need the fact, which we will prove by degeneration,
that the rational curve $C\subset X$ we start with meets the singular fibers of $X/\PP^1$ transversely.

Given an elliptic surface $X_{\Gamma}\to \Gamma$, we
let $\Sigma \subset \Gamma$ correspond to the singular fibers of 
$\rho_{\Gamma}:X_{\Gamma}\to \Gamma$, with inclusion $j : U:= \Gamma\bs \Sigma \hookrightarrow
\Gamma$. So we have a diagram:
\[
\begin{matrix}
X_U&\hookrightarrow&X_{\Gamma}\\
&\\
\rho_U\biggl\downarrow\quad&&\quad\biggr\downarrow\rho_{\Gamma}\\
&\\
U&{\buildrel j\over\hookrightarrow}&\Gamma,
\end{matrix}
\]
where $\rho_U$ is smooth and proper.  {We consider the $i$-th Leray 
direct image sheaves $R^i\rho_{U,*}\C$ and $R^i\rho_{\Gamma,*}\C$. 
The sheaf of invariant ``cycles'', i.e. those
$i$-th cohomology classes in the fibers of $\rho_U$ that are invariant under local monodromy,
is given by $ j_{\ast}R^i\rho_{U,\ast}\C$.}
The local invariant cycle property (see \cite{Z2}, \S15) gives us a surjection:
\[
R^i\rho_{\Gamma,\ast}\C \to j_{\ast}R^i\rho_{U,\ast}\C,
\]
for all $i$ and hence
\[
H^1(\Gamma,R^1\rho_{\Gamma,\ast}\C) \simeq H^1(\Gamma,j_{\ast}R^1\rho_{U,\ast}\C).
\]
The Leray spectral sequence for $\rho_{\Gamma}$ degenerates at $E_2$ (see \cite{Z2}, \S15).
This is induced by a Leray filtration: $H^2(X_{\Gamma},\Q) = $
\[
L^0H^2(X,\Q) \supset L^1H^2(X_{\Gamma},\Q) 
\supset L^2H^2(X_{\Gamma},\Q) \supset L^3H^2(X_{\Gamma},\Q) = 0.
\]
Let $Gr_L^i H^2(X_{\Gamma},\Q) = L^iH^2(X_{\Gamma},\Q)/L^{i+1}H^2(X_{\Gamma},\Q)$. Note that 
\[
Gr_L^2H^2(X_{\Gamma},\Q) = L^2H^2(X_{\Gamma},\Q) =
H^2(\Gamma,R^0\rho_{\Gamma,\ast}\Q) = \Q[F] \simeq \Q,
\]
where  we use the fact that $R^0\rho_{\Gamma,\ast}\Q \simeq \Q$ is the constant
sheaf. Further,
\[
Gr_L^1H^2(X_{\Gamma},\Q) = 
H^1(\Gamma,R^1\rho_{\Gamma,\ast}\Q) \simeq H^1(\Gamma,j_{\ast}R^1\rho_{U,\ast}\Q),
\]
 and the kernel  of the surjective map
\[
H^2(X_{\Gamma},\Q) \twoheadrightarrow Gr_L^0H^2(X_{\Gamma},\Q )= H^0(\Gamma,R^2\rho_{\Gamma,\ast}\Q),
\]
defines $L^1H^2(X_{\Gamma},\Q)$.
There are short exact sequences:
\[
0 \to  \Q[F]  \to L^1H^2(X_{\Gamma},\Q) \to H^1(\Gamma,j_{\ast}R^1\rho_{U,\ast}\Q)\to 0,
\]
\[
0\to H^1(\Gamma,j_{\ast}R^1\rho_{U,\ast}\Q)
 \to \frac{H^2(X_{\Gamma},\Q)}{\Q[F]} \to H^0(\Gamma,R^2\rho_{\Gamma,\ast}\Q)
\to 0.
\]
There is a commutative diagram
\begin{equation}\label{E27}
\begin{matrix}
0&\to& H^1(\Gamma,j_{\ast}R^1\rho_{U,\ast}\Q)&
& \to& \frac{H^2(X_{\Gamma},\Q)}{\Q[F]}& \to& H^0(\Gamma,R^2\rho_{\Gamma,\ast}\Q)
&\to& 0\\
&\\
&&&&&\rho_{\Gamma,\ast}\biggl\downarrow\quad&&\quad\biggr\downarrow
\ol{\rho}_{\Gamma,\ast}\\
&\\
&&&&&H^0(\Gamma,\Q)&=&H^0(\Gamma,\Q),
\end{matrix}
\end{equation}
where $\ol{\rho}_{\Gamma,\ast}$ is induced from ${\rho}_{\Gamma,\ast}$.
Note that $\ker\big(\ol{\rho}_{\Gamma,\ast}\big)$ will involve the components
of the bad fibers of $\rho_{\Gamma}$. For instance, if all the fibers of $\rho_{\Gamma}$
are irreducible, then $\ol{\rho}_{\Gamma,*}$ is an isomorphism.
Let $F_t := \rho_{\Gamma}^{-1}(t)$. There are holomorphic vector bundles over $U$:
\[
\F^1 := \CO_{U}\biggl(\coprod_{t\in U}H^{1,0}(F_t,\C)\biggr) \subset \F
:= \CO_{U}\biggl(\coprod_{t\in U}H^{1}(F_t,\C)\biggr),
\]
\[
\quad \F^{1,\ast} :=  \F/\F^1 = \CO_{U}\biggl(\coprod_{t\in U}H^{0,1}(F_t,\C)\biggr), 
\]
with canonical extensions
\[
\ol{\F}^1 \subset \ol{\F}, \quad \ol{\F}^{1,\ast} := \ol{\F}/\ol{\F}^1,
\]
over $\Gamma$ (see \cite{Z2}, \S3), as well as a  short exact sequences of sheaves:
\[
0 \to R^1\rho_{U,\ast}\Z \to {\F}^{1,\ast}\to {\J}\to 0\quad {\rm (over} \ U {\rm )},
\]
\[
0 \to j_{\ast}R^1\rho_{U,\ast}\Z \to \ol{\F}^{1,\ast}\to \ol{\J}\to 0 \quad {\rm (over} \ \Gamma {\rm )},
\]
where $\J,\ \ol{\J}$ are the sheaves of germs of normal functions over $U$ and $\Gamma$
respectively. Apart from playing a role in the limiting behavior
of normal functions about  {the} singular points $\Gamma\bs U$, the canonical extensions are useful
in calculating the Hodge filtration 
$\big\{F^{\ell}H^i(\Gamma,j_{\ast}R^1\rho_{\Gamma,\ast}\C)\big\}_{\ell\geq 0}$
on $H^i(\Gamma,j_{\ast}R^1\rho_{\Gamma,\ast}\C)$.
More specifically, from the work of \cite{Z2},
$H^i(\Gamma,j_{\ast}R^1\rho_{\Gamma,\ast}\Z)$ is naturally endowed with a pure
Hodge structure of weight $i+1$; moreover 
from (\cite{Z2}, \S9), one has isomorphisms:
\[
H^1(\Gamma,\ol{\F}^{1,\ast}) \simeq \frac{H^1(\Gamma,j_{\ast}R^1\rho_{U,\ast}\C)}{
F^1H^1(\Gamma,j_{\ast}R^1\rho_{U,\ast}\C)},
\]
\[
H^0(\Gamma,\ol{\F}^{1,\ast}) \simeq \frac{H^0(\Gamma,R^1\rho_{\Gamma,\ast}\C)}{
F^1H^0(\Gamma,R^1\rho_{\Gamma,\ast}\C)}.
\]
(It is worthwhile pointing out
that outside of cases of trivial $j$-invariant,  one has
$H^0(\Gamma,R^1\rho_{\Gamma,\ast}\C) = 0$ (see \cite{C-Z}, p. 5).)
Taking cohomology, one has a short exact sequence:
\[
0\to \frac{H^0(\Gamma,\ol{\F}^{1,\ast})}{ H^0(\Gamma,j_{\ast}R^1\rho_{U,\ast}\Z)} \to
H^0(\Gamma,\ol{\J}) \xrightarrow{\delta} H^1(\Gamma,j_{\ast}R^1\rho_{U,\ast}\Z)^{1,1} \to 0,
\]
where $H^1(\Gamma,j_{\ast}R^1\rho_{U,\ast}\Z)^{1,1} :=$
\[
 \ker \big(H^1(\Gamma,j_{\ast}R^1\rho_{U,\ast}\Z) 
\to H^1(\Gamma,j_{\ast}R^1\rho_{U,\ast}\C)\big/F^1H^1(\Gamma,j_{\ast}R^1\rho_{U,\ast}\C)\big).
\]
{The term $H^1(\Gamma,j_{\ast}R^1\rho_{U,\ast}\Z)^{1,1} $ admits
the following interpretation \footnote{At least in the situation of  the setting
of \cite{T}, where there are no exceptional curves of the first kind in the fibers of $\rho_{\Gamma}$,
$X_{\Gamma}$ admits a section, and nonconstant $j$-invariant.}:}  If we tensor $H^1(\Gamma,j_{\ast}R^1\rho_{U,\ast}\Z)^{1,1}$
with $\Q$, apply diagram (\ref{E27}), (which is well-known to be a diagram of Hodge structures),
{and restrict to algebraic cocycles,}
we arrive at the short exact sequence:
\begin{equation}\label{E27.7}
0 \to H^1(\Gamma,j_{\ast}R^1\rho_{U,\ast}\Q)^{1,1} \to 
\frac{H^{2}_{\rm alg}(X_{\Gamma},\Q)}{\Q[F]} \to H^0(\Gamma,R^2\rho_{\Gamma,\ast}\Q)
\to 0,
\end{equation}
where $H^2_{\rm alg}(X,\Q)\subset H^2(X,\Q)$ is the subspace of algebraic
cocycles.  {The intersection pairing involving the components of the bad
fibers of $\rho_{\Gamma}$ is well understood in terms of a negative definite property (see Lemma 1.3 of \cite{T}).
In particular  from (\ref{E27.7}) one can argue that
$H^1(\Gamma,j_{\ast}R^1\rho_{U,\ast}\Q)^{1,1}$ can be identified with
the quotient of the N\'eron-Severi group of $X$ (over $\Q$) by the subgroup generated by 
the irreducible components of the bad  fibers of $\rho_{\Gamma}$.}
The group $H^0(\Gamma,\ol{\J})$ is called the group of normal functions, and for
$\nu\in H^0(\Gamma,\ol{\J})$, $\delta(\nu) \in H^1(\Gamma,j_{\ast}R^1\rho_{U,\ast}\Z)^{1,1}$
is called its topological invariant.  We say that $\delta(\nu)$ is nontorsion if
$\delta(\nu) \ne 0$ as a class in $H^1(\Gamma,j_{\ast}R^1\rho_{U,\ast}\Q)^{1,1}$.\footnote{{Taking
into account Remark 2.4 of  \cite{T}  (and again assuming that
$X_{\Gamma}$ satisfy the assumptions stated in the previous footnote),  $H^1(\Gamma,j_{\ast}R^1\rho_{U,\ast}\Z)$ is
torsion free, hence the same holds for $H^1(\Gamma,j_{\ast}R^1\rho_{U,\ast}\Z)^{1,1}
= H^1(\Gamma,j_{\ast}R^1\rho_{U,\ast}\Z) \cap F^1H^1(\Gamma,j_{\ast}R^1\rho_{U,\ast}\C)$.}}
We need the following key observation:

\begin{prop}\label{P278} Suppose
that $\nu \in H^0(\Gamma,\ol{\J})$ is given such that $\delta(\nu)$ is nontorsion. 
Then for sufficiently general $t\in U$, the cyclic group generated
by $\nu(t)$ is dense in $J^1(E_t)$.
\end{prop}

\begin{proof} A local lifting of the normal function $\nu\big|_U\in H^0(U,\J)$  determines
an analytic function on a disk $\Delta\subset U$, viz., $\tilde{\nu} \in H^0(\Delta,{\F}^{1,\ast})
\simeq H^0(\Delta,\CO_{\Delta})$, using the fact that $\F^{1,\ast}$ is a holomorphic
line bundle. Further, we have the family of lattices $H^0(\Delta,R^1\rho_{U,\ast}\Z) 
\hookrightarrow H^0(\Delta,{\F}^{1,\ast})$. Let $\delta_1,\ \delta_2 \in H^0(\Delta,
R^1\rho_{U,\ast}\Z)$ be generators with respective images 
$[\delta_1],\ [\delta_2]\in H^0(\Delta,\F^{1,\ast})$, under the (injective) composite
\[
H^0(\Delta,R^1\rho_{U,\ast}\Z)  \to H^0(\Delta,\F) \to H^0(\Delta,\F^{1,\ast}).
\] 
Thus we can write
\[
\tilde{\nu}(t) = x(t)[\delta_{1,t}] + y(t)[\delta_{2,t}],
\]
for unique real-valued  functions $x(t),\ y(t)$, $t \in \Delta$.
Note that $[\delta_{2,t}] = g(t)[\delta_{1,t}]$ for some
holomorphic function $g(t)$, and likewise $\tilde{\nu}(t) = h(t)[\delta_{1,t}]$
for a holomorphic $h(t)$.  Thus $h(t)  = x(t) \ +\ y(t)g(t)$, and in particular
${\rm Re}(h(t)) = x(t) + y(t){\rm Re}(g(t))$, ${\rm Im}(h(t)) = y(t){\rm Im}(g(t))$.
Thus $x(t)$and $y(t)$ are real analytic functions. If the 
cyclic group generated by $\nu(t)$ is not dense in $J^1(F_t)$
for uncountably many $t\in \Delta$, then by a countability and Baire type argument together
with {Kronecker's theorem},
$\{1,x(t),y(t)\}$ lie on a hyperplane
$a_1x + a_2y + a_3 = 0$ in $\R^2$, where $\{a_j\} \in \Q$ are constant with respect
to $t\in \Delta$ and not all zero.  Using  $h(t)  = x(t) \ +\ y(t)g(t)$,
one can easily check then that 
$a_1x(t) + a_2y(t) + a_3 = 0$ for all $t\in \Delta$ implies that
$\tilde{\nu}$ is constant.  More precisely, one can choose the lift 
$\tilde{\tilde{\nu}}\in H^0(\Delta,R^1\rho_{U,\ast}\C)$ of $\tilde{\nu}$
via the composite
\[
H^0(\Delta,R^1\rho_{\Delta,\ast}\C) \to H^0(\Delta,\F) \to H^0(\Delta,\F/\F^1) =
H^0(\Delta,\F^{1,\ast}),\ \tilde{\tilde{\nu}}\mapsto \tilde{\nu}.
\]
This tells us that the
Griffiths' infinitesimal invariant of $\nu$ over $U$ (see \cite{G}, p. 69) is zero.
However in this case the Griffiths' infinitesimal invariant is known to coincide
with the topological  de Rham invariant (see \cite{MS},
as well as \cite{L-S} for a background on this). 
In the end, this translates to saying  that $\delta\big(\nu\big|_U\big) = 0$ as a class in
$H^1(U,R^1\rho_{U,\ast}\Q)$. Alternatively and more directly, if we assume for the moment that
the $j$-invariant of the family $X_U\to U$ is nonconstant, then 
$\F^1\cap R^1\rho_{U,\ast}\C = 0 \in \F$ and hence $\nu\big|_U$
is induced by a class in $H^0(U,R^1\rho_{U,\ast}\C/R^1\rho_{U,\ast}\Z)$,
and therefore $\delta\big(\nu\big|_U\big) = 0\in H^1(U,R^1\rho_{U,\ast}\Q)$. The
same conclusion holds, albeit by a tedious argument,  
if the $j$-invariant is constant - the details are left to the
reader and involve a generalization of Example \ref{EX01} below. Next by (\cite{Z2}, \S14), the map
$H^1(\Gamma,R^1\rho_{\Gamma,\ast}\Q)\hookrightarrow H^1(U,R^1\rho_{U,\ast}\Q)$
is injective. Hence $\delta(\nu)  = 0 \in H^1(\Gamma,R^1\rho_{\Gamma,\ast}\Q)$,
a contradiction.
\end{proof}

\begin{ex}\label{EX01} {\rm Let $E$ be an elliptic curve and $Y = E\times E$. 
We can illustrate Proposition \ref{P278} rather easily in this situation. With
regard to the first projection $Y\to E$, the sheaf of germs of normal functions $\J$ is
given by the short exact sequences of sheaves over $E$:
\[
0\to H^1(E,\Z) \to \CO_E\big(H^{0,1}(E)\big) \to \J \to 0.
\]
Note that $H^1\big(E,\CO_E\big(H^{0,1}(E)\big)\big) \simeq H^{0,1}(E)\otimes H^{0,1}(E)$,
and hence there is the short exact sequence:
\[
0 \to J^1(E) \to H^0(E,\J) \xrightarrow{\delta} \big\{H^1(E,\Z)\otimes H^1(E,\Z)\big\}
\bigcap H^{1,1}(Y)\to 0.
\]
If $\nu \in H^0(E,\J)$ has trivial infinitesimal invariant, then
\[
\nu \in H^0\big(E,H^{0,1}(E)/H^1(E,\Z)\big) \simeq J^1(E),
\]
and hence $\delta(\nu) = 0$. For $n\in \BN$, let
$f_n : E\to E$ be given by multiplication by $n$, and 
let $\Xi(n)$ be the  graph
of $f_n$ in $Y$,  with K\"unneth
component $[\Xi(n)^{1,1}] \in H^1(E,\Z)\otimes H^1(E,\Z)$.
It follows rather directly from {Kronecker's theorem} that
\[
\bigcup_{n\in \BN}\Xi(n) \subset Y,
\]
is dense in $Y$ in the strong topology.
 Note however that if $\nu$ is
the normal function associated to $f_1$, then $n\nu$ is
the normal function associated to $f_n$. Furthermore
$\delta(n\nu) = [\Xi(n)^{1,1}] \ne  0$, and hence
the density also follows from Proposition \ref{P278}.
 Now let $X := Y/\pm$ be the
corresponding Kummer counterpart with $C_n$  being the image
of $\Xi(n)$ in $X$. Then $C_n$ is a rational curve and
\[
\bigcup_{n\in \BN}C_n\subset X,
\]
is likewise dense in $X$ in the strong topology.
}
\end{ex}

This gives us the density of
rational curves in the strong topology
on an elliptic $K3$ surface $X$ as long as $\pi: X\to \PP^1$ admits a rational nt multisection, i.e., \thmref{THM004}.
It also gives another proof of \propref{PROP001} as long as we can find a rational nt multisection.

On the other hand, we can produce such a
multisection using our deformational argument as follows. Although this is redundant, we keep it here to make
our paper self-contained.

Now let us consider a general elliptic $K3$ surface $(X, L)\in \CP_{g,m}$ with
$\rho: X \to \PP^1$ the elliptic fibration given by $|F|$.

Let $\Gamma_0 \in |L|$ be a rational curve, with desingularization $\Gamma = \PP^1$.
Note that $\rho\big|_{\Gamma_0}:\Gamma_0 \to \PP^1$ has degree $m$, and
hence the corresponding map $\lambda : \Gamma \to \PP^1$ is of degree $m$.
Base change gives us an elliptic surface $\rho_{\Gamma} : X_{\Gamma} \to \Gamma$,
with section $\sigma : \Gamma \hookrightarrow X_{\Gamma}$, (where we can assume after
a proper modification, that $X_\Gamma$ is smooth).
Let $h : X_{\Gamma} \to X$
be the obvious morphism (of degree $m$). Note that 
\[
h_{\ast}(\sigma(\Gamma)) = \Gamma_0;
\]
moreover we have corresponding classes $F$, $h^{\ast}(L)$ on $X_{\Gamma}$,
with $h_{\ast}(F) = F$, and on $X_{\Gamma}$:
\[
F^2 = 0,\ (h^{\ast}(L))^2 = m\cdot (2g-2), \ (\sigma(\Gamma))^2 = b,\ ({\rm some}\ b \in \Z),
\]
\[
F\cdot \sigma(\Gamma) = 1, \  F\cdot h^{\ast}(L) = m,\ \sigma(\Gamma)\cdot h^{\ast}(L) = 2g-2.
\]
Note that  $\{F,h^{\ast}(L),\sigma(\Gamma)\}$ are independent over $\Q$
iff $m\cdot b \ne 2g-2$. The independence follows from
\begin{lem}
$b < 0$.
\end{lem}
\begin{proof} Since $\sigma(\Gamma) = \PP^1$, the adjunction formula
tells us that
\[
-2 = b + K_{X_{\Gamma}}\cdot \sigma(\Gamma).
\]
But $h$ is ramified only along the fibers of $\rho_{\Gamma}$, i.e. over which $\lambda$
ramifies, and hence
\[
K_{X_{\Gamma}} = h^{\ast}(K_X) + k\cdot F,
\]
for some integer $k\geq 0$. But $X$ a $K3$ surface implies that $K_X = 0$,
and hence $b = -(2+k) < 0$.
\end{proof}

Now let us suppose that:
\begin{equation}\label{E301}
\ol{\rho}_{\Gamma,\ast}\ {\rm in\ diagram}\ (\ref{E27})\ {\rm is\ an\ isomorphism.}
\end{equation}
Then using the fact that $\rho_{\Gamma,\ast}(
h^{\ast}(L)) = m\Gamma$,  by (\ref{E27}) it follows that 
\[
[h^{\ast}(L) - m\sigma(\Gamma)] \ne 0\in H^1(\Gamma,j_{\ast}R^1\rho_{\Gamma,\ast}\Q)^{1,1}.
\]

The general story (viz., when (\ref{E301}) is not satisfied)
 involves a rational linear combination
of the components of the bad fibers of $\rho_{\Gamma}$ together with
$[h^{\ast}(L) - m\sigma(\Gamma)]$ (compare for example (\cite{C-Z}, thm 1.6)). 
The argument in showing  that this gives a nontrivial class in 
$H^1(\Gamma,j_{\ast}R^1\rho_{\Gamma,\ast}\Q)^{1,1}$
is similar but more complicated.
For our purpose, we can always choose $\Gamma_0$ such that it meets the singular fibers of $X/\PP^1$ transversely
and hence $X_{\Gamma}$ is smooth and has irreducible fibers over $\Gamma$ and \eqref{E301} is trivially
satisfied.
Then $[h^{\ast}(L) - m\sigma(\Gamma)]$ determines a normal function $\nu$
with
\[
\delta(\nu) = [h^{\ast}(L) - m\sigma(\Gamma)] \in H^1(\Gamma,j_{\ast}R^1\rho_{\Gamma,\ast}\Z)^{1,1}
\]
and hence by Proposition \ref{P278}, $\nu(t)$ has nontrivial dynamics for general $t\in \Gamma$.

It remains to verify the following.

\begin{lem}\label{LEM005}
For all $g,m\in \BZ^+$ satisfying $g\ge 2$ and $2|m$ and
a general $(X,L)\in \CP_{g,m}$, there is an irreducible nodal rational curve in $|L|$ that meets all singular
curves in $|F|$ transversely.
\end{lem}

\begin{proof}
It is well known that $X/\PP^1$ has $24$ nodal fibers. It suffices to figure out where these $24$ curves in
$|F|$ go when we degenerate $X$.
Let $\CS/\Delta$ be the family of $K3$ surfaces constructed in \ssecref{SSECDEFORM}. A curve $N\in |F|$ on
the central fiber $\CS_0 = R = R_1 \cup R_2$ is described in \remref{REM003}. It is not hard to see that $N$
is a limit of nodal rational curves in $|F|$ on the general fibers if one of the following holds:
\begin{itemize}
\item $N$ passes through one of the sixteen rational double points $p_j$
and there are sixteen such curves;
\item $N$ passes through one of the four points $\{ p\in D: G_1 \sim_{\rat} 2p \text{ on } D\}$
and there are four such curves;
\item $N$ passes through one of the four points $\{ p\in D: G_2 \sim_{\rat} 2p \text{ on } D\}$
and there are four such curves.
\end{itemize}
One can check that these add up to $24$.

Now we let $C = C_1 \cup C_2$ with irreducible components $C_i \in |L_i|$ satisfying
\begin{equation}\label{E302}
C_1\cdot D = C_2 \cdot D = (g+1) q
\end{equation}
for some point $q\in D$. This is a limiting rational curve and it obviously meets each of the $24$ curves
$N\in |F|$ given above transversely.
\end{proof}

\section{An open question}

Suppose that $X$ is a $K3$ surface defined over $\ol{\Q}$. Let $\Sigma\subset X$
be the union of all rational curves on $X$. Rigidity arguments  imply that every rational
curve in $X$ is defined over $\ol{\Q}$.
The following was raised by Matt Kerr (\cite{Ke}):

\begin{q}{\rm  Is $X(\ol{\Q})\subset \Sigma(\ol{\Q})$?}
\end{q}

An affirmative answer to this question would not only imply that
$\Sigma$ is dense in $X(\C)$ in the usual topology,
but this would also provide a nontrivial instance of the
Bloch-Beilinson conjecture on the injectivity of Abel-Jacobi maps for 
smooth projective varieties
defined over $\ol{\Q}$. More specifically, by an application of the
connectedness part of Bertini's theorem, $\Sigma$ is connected, 
hence  $\CH_{\hom}^2(X/\ol{\Q}) = 0$.

It was pointed out to us by the referee that this question actually has a long, albeit poorly
documented  history. It was posed by F. Bogomolov as far back as 1981 \cite{BT2}.


\begin{thebibliography}{99}
\bibitem[BPV]{BPV} W. Barth, C. Peters and A. van de Ven,
{\it Compact Complex Surfaces\/}, Springer-Verlag, Berlin, 1984.

\bibitem[BHT]{BHT} F. Bogomolov, B. Hassett and Y. Tschinkel,
Constructing rational curves on $K3$ surfaces, to appear in {\it Duke Math. J.\/}, preprint arXiv:0907.3527.

\bibitem[B-L]{B-L} J. Bryan and N. C. Leung,
The Enumerative geometry of $K3$ surfaces and modular Forms,
{\it J. Amer. Math. Soc.\/} {\bf 13} (2000), no. 2, 371-410.
Also preprint alg-geom/9711031.

\bibitem[Bu-L]{Bu-L} G. Buzzard and S. Y. Lu,
Algebraic surfaces holomorphically dominable by $\BC^2$,                                            
{\it Invent. Math.\/} {\bf 139} (2000), no. 3, 617-659.

\bibitem[BT1]{BT1} F. A. Bogomolov and Y. Tschinkel,
Density of rational points on elliptic $K3$ surfaces,
{\it Asian J. Math.\/} {\bf 4} (2000), no. 2, 351-368.

\bibitem[BT2]{BT2}
F. A. Bogomolov and Y. Tschinkel, Rational curves and points on $K3$
surfaces, {\it Amer. J. Math.\/} {\bf 127} (2005), no. 4, 825-835.

\bibitem[C]{C} X. Chen, Rational curves on $K3$ surfaces,
{\it J. Alg. Geom.\/} {\bf 8} (1999), 245-278. Also preprint math.AG/9804075.

\bibitem[C-L]{C-L}
X. Chen and J. D. Lewis,
The Hodge-${\mathcal D}$-conjecture for $K3$ and Abelian surfaces,
{\it J. Alg. Geom.} {\bf 14} (2005), 213-240.

\bibitem[CLM]{CLM} C. Ciliberto, A. Lopez and R. Miranda, Projective
degenerations of $K3$ surfaces, Gaussian maps, and Fano threefolds,
{\it Invent. Math.\/} {\bf{114}}, 641-667 (1993).

\bibitem[C-Z]{C-Z} D. Cox and S. Zucker, Intersection numbers of 
sections of elliptic surfaces, {\em Inventiones math.} {\bf 53}, (1979), 1-44.

\bibitem[D]{D} T. Dedieu, Severi varieities and self rational maps of $K3$ surfaces,
{\it Internat. J. Math.\/} {\bf 20} (2009), no. 12, 1455-1477. Also preprint arXiv:0704.3163.

\bibitem[H-T]{H-T} J. Harris and Y. Tschinkel, Rational points on quartics,
{\it Duke Math. J.\/} {\bf 104} (2000), no. 3, 477-500.

\bibitem[G]{G} M. Green, {\em Infinitesimal methods in Hodge theory}, in
Algebraic Cycles and Hodge Theory (Torino, 1993), 1-92, Lecture Notes in Math., 
{\bf 1594}, Springer, Berlin, 1994.

\bibitem[Gr]{Gr} P. A. Griffiths, On the periods of certain rational integrals: I
{\em  Annals of Mathematics,} Second Series, Vol. {\bf 90}, No. 3 (Nov., 1969), 460-495.

\bibitem[G-H]{G-H} P. A. Griffiths and J. Harris, 
On the Noether-Lefschetz theorem and some remarks on codimension two cycles,
{\it Math. Ann.\/} {\bf 271} (1985), no. 1, 31-51.

\bibitem[H-W]{H-W} G. H. Hardy and E. M. Wright, {\em An introduction to the theory of numbers},
Fifth edition. The Clarendon Press, Oxford University Press, New York, 1979. xvi+426 pp. ISBN: 0-19-853170-2; 0-19-853171-0.

\bibitem[H]{H}
B. Hassett,
Potential density of rational points on algebraic varieties,
{\em Higher dimensional varieties and rational points (Budapest, 2001)},
223-282, Bolyai Soc. Math. Stud., 12, Springer, Berlin, 2003.

\bibitem[Ke]{Ke} M. Kerr. Personal communication.

\bibitem[K]{K} S. Kobayashi,
{\it Hyperbolic Manifolds and Holomorphic mappings\/}, Marcel Dekker,
New York, 1978.

\bibitem[La]{La} S. Lang, Hyperbolic and Diophantine analysis,
{\it Bull. Amer. Math. Soc. (N.S.)} {\bf 11} (1986),
no. 2, 159-205.

\bibitem[L-S]{L-S} J. D. Lewis and S. Saito,
Algebraic cycles and Mumford-Griffiths invariants, 
{\em Amer. J. Math.} {\bf 129}, (2007), no. 6, 1449-1499.

\bibitem[L-L]{L-L} J. Li and C. Liedtke, Rational curves on $K3$ surfaces,
preprint arXiv:1012.3777.

\bibitem[M-P]{M-P} D. Maulik and R. Pandharipande, Gromov-Witten theory and Noether-Lefschetz theory,
preprint arXiv:0705.1653.

\bibitem[M-M]{M-M} S. Mori and S. Mukai, {\it The Unirulesness of the
Moduli Space of Curves of Genus 11\/}, Lecture Notes in Mathematics,
vol. 1016 (1982), 334-353.

\bibitem[MS]{MS} M. Saito,
{\em Direct image of logarithmic complexes and infinitesimal invariants of cycles,}
in Algebraic Cycles and Motives. Vol. 2, 304-318, London Math. Soc. Lecture Note Ser., {\bf 344}, Cambridge Univ. Press, Cambridge, 2007.

\bibitem[T]{T} T. Shioda, {\it On elliptic modular surfaces,}. J. Math. Soc. Japan {\bf 24}(1), (1972), 20-59.

\bibitem[Z1]{Z1} S. Zucker, Generalized intermediate jacobians and the
theorem on normal functions, {\em Inventiones Math.} {\bf 33}, (1976), 185-222.

\bibitem[Z2]{Z2} S. Zucker, Hodge theory with degenerating coefficients: $L_2$
cohomology in the Poincar\'e metric,  {\em Annals of Math.} {\bf 109}, (1979),
415-476.
\end{thebibliography}
\end{document}